\documentclass[12pt]{amsart}

\usepackage{latexsym,amsmath,amssymb,amsthm,amsfonts}

\usepackage[numeric]{amsrefs}

\usepackage[capitalise]{cleveref}

\usepackage{multido,pst-plot,pstricks,pst-node}
\newcommand{\R}{{\mathbb{R}}}
\newcommand{\eps}{\varepsilon}

\newcommand{\E}{{\mathbb{E}}}

\renewcommand{\P}{{\mathcal{P}}}

\newcommand{\A}{{\mathcal{A}}}
\newcommand{\B}{{\mathcal{B}}}
\newcommand{\x}{{\boldsymbol{x}}}

\renewcommand{\u}{{\boldsymbol{u}}}

\newcommand{\e}{{\boldsymbol{e}}}
\newcommand{\y}{{\boldsymbol{y}}}
\renewcommand{\a}{{\boldsymbol{a}}}
\renewcommand{\b}{{\boldsymbol{b}}}

\newcommand{\intr}{{\rm int}}

\renewcommand{\phi}{{\boldsymbol{\varphi}}}

\newcommand{\qtq}[1]{{\quad\text{#1}\quad}}

\newtheorem{theorem}{Theorem}[section]
\newtheorem{lemma}[theorem]{Lemma}
\newtheorem{proposition}[theorem]{Proposition}

\theoremstyle{remark}

\numberwithin{equation}{section}

\title{On the Hadwiger covering problem in low dimensions}

\author{A.\ Prymak}
\address{Department of Mathematics, University of Manitoba, Winnipeg, MB, R3T 2N2, Canada}
\email{prymak@gmail.com}
\thanks{Corresponding author: Andriy Prymak; email: {\tt prymak@gmail.com} \\ \phantom{fff} The first author was supported by NSERC of Canada Discovery Grant RGPIN-2020-05357.}

\author{V.\ Shepelska}
\address{Department of Mathematics, University of Manitoba, Winnipeg, MB, R3T 2N2, Canada}
\email{shepelska@gmail.com}
\thanks{The second author was partially supported by the PIMS Postdoctoral Fellowship.}

\keywords{Illumination problem, illumination number, covering number, covering by smaller homothetic copies, convex body, 4-cube}

\subjclass[2010]{Primary 52A20; Secondary 52A37, 52A40, 52C17}

\begin{document}

\begin{abstract}
Let $H_n$ be the minimal number of smaller homothetic copies of an $n$-dimensional convex body required to cover the whole body. Equivalently, $H_n$ can be defined via illumination of the boundary of a convex body by external light sources. The best known upper bound in three-dimensional case is $H_3\le 16$ and is due to Papadoperakis. We use Papadoperakis' approach to show that $H_4\le 96$, $H_5\le 1091$ and $H_6\le 15373$ which significantly improve the previously known upper bounds on $H_n$ in these dimensions.
\end{abstract}

\maketitle

\section{Introduction and results}

Let $\E^n$ denote the $n$-dimensional Euclidean space. A convex body in $\E^n$ is a convex compact set having non-empty interior. For two sets $A,B\subset \E^n$ we let $C(A,B)$ be the smallest number of translates of $B$ required to cover $A$, and let $\intr(A)$ denote the interior of $A$.

Hadwiger~\cite{Ha} asked what is the smallest value $H_n$ of $C(K,\intr(K))$ for arbitrary convex body $K$ in $\E^n$. This is equivalent to the question about the least number of smaller homothetic copies of $K$ which are able to cover $K$, and, as was shown by Boltyanski~\cite{Bo}, to the question about the smallest number of external light sources required to illuminate the boundary of every convex body. Considering cube, one immediately gets $H_n\ge 2^n$. The related primary conjecture, which is commonly referred to as Hadwiger conjecture or as Gohberg-Markus covering conjecture, is that $H_n=2^n$, but this is known (and is simple) only for $n=2$. Below we give a brief overview of the known results about $H_n$. For a detailed history of the question and survey including many partial results for special classes of convex bodies see, e.g., \cite{Be}.

The best known explicit upper bound on $H_n$ in high dimensions is a combination of Rogers's result~\cite{Ro} on the covering density of $\E^n$ by translates of arbitrary convex body with the Rogers-Shephard inequality~\cite{Ro-Sh} bounding the volume of the difference body. The interested reader is referred to~\cite{Be}*{Section~2.2} for further details and related results. Here we only state the actual bound, which is
\begin{equation}\label{eqn:asymptotic}
H_n\le \binom{2n}{n}n(\ln n+\ln\ln n+5),
\end{equation}
where $5$ can be replaced by $4$ for sufficiently large $n$. %Recently, Huang, Slomka, Tkocz and Vritsiou obtained a remarkable asymptotic improvement of this estimate in~\cite{HSTV}:
%\begin{equation*}
%H_n\le \binom{2n}{n} e^{-c\sqrt{n}},
%\end{equation*}
%where $c>0$ is an implicit constant. 
Lassak~\cite{La} showed that
\begin{equation}\label{eqn:lassak}
H_n\le (n+1)n^{n-1}-(n-1)(n-2)^{n-1},
\end{equation}
which is better than~\eqref{eqn:asymptotic} for $n\le 5$ and up to now was the best known bound for $n=4,5$.
In~\cite{Pa} Papadoperakis showed that $H_3\le 16$, which is the best known bound in three dimensions.

The key idea of~\cite{Pa} is to reduce the problem of Hadwiger to that of covering specific sets of relatively simple structure by certain rectangular parallelotopes. Namely, we have $H_n\le C_n$, where $C_n$ is a related covering number which will be introduced in \cref{sec:Pap}, see~\eqref{eqn:C_n}. In these terms, it was shown in~\cite{Pa} that $H_3\le C_3\le 16$. In fact, it is not hard to prove the estimate in the other direction and establish that $C_3=16$, so $16$ is the best one can get with this method for three dimensions.

Our goal is to obtain upper bounds of $C_n$ for $n\ge4$. We begin with the bound for $n\ge 5$.

\begin{proposition}\label{thm:general n}
	For any $n\ge 5$ we have
	$C_n \le 2n(n-1)(n-2)^{n-2}+2n+1$.
\end{proposition}

We obtain much sharper estimate for $n=4$.

\begin{proposition}\label{thm:main}
	For the four-dimensional case we have $C_4\le 96$.
\end{proposition}

As an immediate corollary of the above two propositions, since $H_n\le C_n$, we obtain the following new upper bounds on $H_n$ for $n=4,5,6$, which is the main result of this paper.

\begin{theorem}\label{cor:main}
	$H_4\le 96$, $H_5\le 1091$, $H_6\le 15373$.
\end{theorem}

Below we provide a table comparing various upper estimates on $H_n$ stated above, which shows that our estimates for $H_4$, $H_5$, and $H_6$ are roughly one third of the previously known results.
\vskip3mm
\begin{center}
\begin{tabular}{| c | c | c | c |}
	\hline
	$n$ & \cref{cor:main} &  \eqref{eqn:lassak} & \eqref{eqn:asymptotic} \\
	\hline
	$4$ & $96$ & $296$ & $1879$ \\
	\hline
	$5$ & $1091$ & $3426$ & $8927$ \\
	\hline
	$6$ & $15373$ & $49312$ & $40886$ \\
%	\hline
%	$7$ & $262515$ & $847442$ & $182862$ \\
	\hline
\end{tabular}
\end{center}
\vskip3mm

We will show in a forthcoming work that  $C_n\ge 4n^{n-2}+2n$ for $n\ge 5$ and that $C_4\ge 95$. This means that one cannot achieve an asymptotic improvement of the upper bound on $H_n$ using the connection $H_n\le C_n$ and that the result of \Cref{thm:main} cannot be improved by more than $1$. We believe that $C_4=95$, but proving this may require too much effort to justify such a small improvement, recall that the conjectured value of $H_4$ is $16$.

\Cref{thm:general n} is proved in \cref{sec:general upper}, while \Cref{thm:main} is proved in \cref{sec:4 upper}. %\Cref{cor:main} is an immediate corollary of \cref{thm:general n,thm:main,prop:relation}.

\section{Papadoperakis' reduction to covering problem}\label{sec:Pap}

We extend the notation $C(\cdot,\cdot)$ to the following two situations. For a set $A\subset \E^n$ and families $\A,\B$ of subsets of $\E^n$, we let $C(A,\B)$ be the smallest number of translates of elements of $\B$ required to cover $A$ and let $C(\A,\B)=\sup_{A\in\A} C(A,\B)$ be the smallest number of translates of elements of $\B$ needed to cover an arbitrary element of $\A$.

Let $B_{k,n}$ be the $k$-skeleton of the unit cube $[0,1]^n$, i.e., the union of all $k$-dimensional faces of $[0,1]^n$, or, in other words, the set of all points of the cube having at least $n-k$ coordinates equal to either $0$ or $1$.

Let $\e_i$ denote the $i$-th basic unit vector in $\E^n$, $1\le i\le n$. Note that if a point $\x\in\E^n$ satisfies $\x,\x+\e_i\in[0,1]^n$, then $\x$ and $\x+\e_i$ are necessarily on some opposite $(n-1)$-dimensional faces of the cube $[0,1]^n$ (equivalently, the $i$-th coordinate of these points is $0$ and $1$, respectively). In the collection of sets
\begin{equation*}
\A_n:=\left\{\bigcup_{i=1}^n\{\x_i,\x_i+\e_i\}: \x_i,\x_i+\e_i\in [0,1]^n \right\}
\end{equation*}
each set consists of at most $2n$ points on the boundary of $[0,1]^n$.

We will use the following two families of $n$-dimensional rectangular parallelotopes:
\begin{align*}
\P_n & :=\left\{\prod_{i=1}^n[x_i,x_i+\delta_i]: x_i\in\R, \delta_i\ge0, \sum_{i=1}^n\delta_i<1 \right\} \qtq{and} \\
\P_n^* & :=\left\{\prod_{i=1}^n[x_i,x_i+\delta_i]: x_i\in\R, \delta_i\ge0, \sum_{i=1}^n\delta_i\le 1 \right\}.
\end{align*}
Note that degeneration $\delta_i=0$ in some coordinates is allowed. For simplicity, we will refer to rectangular parallelotopes as boxes.

If  $\A$ is a collection of subsets of $\E^n$ and $B\subset\E^n$, we let $\A\cup B:=\{A\cup B: A\in\A \}$. Finally, we are ready to define the needed covering number:
\begin{equation}\label{eqn:C_n}
C_n:=C(\A_n \cup B_{n-2,n}, \P_n).
\end{equation}

In the above notations, it is established in \cite{Pa}*{Lemmas~1-4}  that $H_3\le C_3$. It turns out that the same arguments can be used in $\E^n$, $n\ge3$, to prove the following.

\begin{proposition}\label{prop:relation}
	$H_n\le C_n$ for any $n\ge3$.
\end{proposition}

We only include an outline of the Papadoperakis' approach since the exposition in~\cite{Pa} is very concise and the generalization to the higher dimensions is straightforward.

First, the shadow $s(\u,X)=\{t\u+{\x}: t>0, {\x}\in X\}$ of a set $X$ when the light comes from the direction $\u$ is defined. Then illumination of the boundary of a convex body is equivalent to covering the boundary of the body by shadows of its interior (\cite{Pa}*{Lemma~1(e)}). Next, the parallelotope $P$ of smallest volume containing a given body $A$ is considered. Note that using affine transformations, one can assume that $P$ is a unit cube. Minimality of the volume of $P$ implies that the tangency points of the faces of $P$ with $A$ can be chosen in such a way that the pairs of points in the opposite faces are different by a unit vector (\cite{Pa}*{Lemma~3}). (In our terminology this means that the set of tangency points belongs to $\A_3$.) In turn, this implies that the interior of $A$ contains a translate of any box from $\P_3$. Finally, a combination of \cite{Pa}*{Lemmas~2 and~4} yields that covering a one-dimensional skeleton of the unit cube together with the tangency points by $m$ boxes from $\P_3$ implies that the whole $P$ and, in particular, the boundary of $A$ can be covered by $m$ shadows of $\intr(A)$, and hence $H_3\le C_3$. As was already mentioned, the same proof works for higher dimensions, and to arrive at our terminology one only needs to note that the union of the relative boundaries of the $(n-1)$-dimensional faces of $[0,1]^n$ is precisely $B_{n-2,n}$, the $(n-2)$-skeleton of the unit cube.

\section{Proof of \Cref{thm:general n}}\label{sec:general upper}

For $\eps>0$ and $A\subset\E^n$, the $\eps$-neighborhood of $A$ is the set of all points $\x\in\E^n$ such that for some $\y\in A$ the distance between $\x$ and $\y$ is less than $\eps$. By a neighborhood of $A$ we mean the $\eps$-neighborhood of $A$ for some $\eps>0$.

\begin{lemma}\label{lem:eps neighborhood}
	If $A\subset\E^n$ is covered by a finite number of boxes from $\P_n$, then each box can be modified so that a neighborhood of $A$ is covered by the resulting boxes while each new box is still from $\P_n$.
\end{lemma}
\begin{proof}
	We replace each box $\prod_{i=1}^n[x_i,x_i+\delta_i]$ with $\prod_{i=1}^n[x_i-\eps,x_i+\delta_i+\eps]$, where $\eps=\tfrac1{3n}(1-\sum_{i=1}^n\lambda_i)$ and depends on the specific box.
\end{proof}

\begin{lemma}\label{lem:compression along one dir}
	Suppose that the box $\prod_{i=1}^n[y_i,y_i+\gamma_i]$ is the union of a finite number of boxes from $\P_n^*$. Then for any $k$ and for any $0<\eps<\gamma_k$ each box can be modified so that the union of the resulting boxes is
	\begin{equation*}
	\left(\prod_{i=1}^{k-1}[y_i,y_i+\gamma_i]\right)\times
	[y_k+\eps,y_k+\gamma_k] \times
	\left(\prod_{i=k+1}^{n}[y_i,y_i+\gamma_i]\right)
	\end{equation*}
	while each new box is from $\P_n$.
\end{lemma}
\begin{proof}
	The idea is to linearly compress the whole structure along the $k$-th coordinate. Let $l:\R\to\R$ be the linear function satisfying $l(y_k+\gamma_k)=y_k+\gamma_k$ and $l(y_k)=y_k+\eps$ whose slope is clearly between $0$ and $1$. Now we simply replace each box $\prod_{i=1}^n[x_i,x_i+\delta_i]$ from the original union by
	\begin{equation*}
	\left(\prod_{i=1}^{k-1}[x_i,x_i+\delta_i]\right)\times
	[l(x_k),l(x_k+\delta_k)] \times
	\left(\prod_{i=k+1}^{n}[x_i,x_i+\delta_i]\right).
	\end{equation*}	
\end{proof}

%\begin{proof}
	Now we prove the required bound $C_n\le 2n(n-1)(n-2)^{n-2}+2n+1$.
	Obviously, any $A\in\A_n$ can be covered by $2n$ elements of $\P_n$, so $C_n=C(\A_n\cup B_{n-2,n},\P_n)\le 2n+ C(B_{n-2,n},\P_n)$ and it remains to show that $C(B_{n-2,n},\P_n)\le m+1$, where $m=2n(n-1)(n-2)^{n-2}$.
	Our strategy is to cover $B_{n-2,n}$ by $m$ elements of $\P_n^*$ first, then add one more box and modify the cover so that all boxes belong to $\P_n$.
	
	Each $(n-2)$-dimensional face $F$ of $[0,1]^n$ is the unit $(n-2)$-dimensional cube which is the union of $(n-2)^{n-2}$ $(n-2)$-dimensional cubes with the side length $\tfrac1{n-2}$ that clearly belong to $\P_n^*$. Since there are $2n(n-1)$ faces of dimension $n-2$, we obtain $m$ boxes from $\P_n^*$ that cover $B_{n-2,n}$. Let $\{P_i\}_{i=1}^m$ be the collection of all such boxes.
	
	We will find one more box $P_0\in\P_n$ and will describe a sequence of steps which modify the boxes from the collection $\{P_i\}_{i=0}^m$. At every step, we will have that $P_i\in\P_n^*$, $0\le i\le m$,
	\begin{equation}\label{eqn:B covering condition n}
	B_{n-2,n}\subset \bigcup_{i=0}^m P_i
	\qtq{and}
	B\subset \bigcup_{0\le i\le m,\, P_i\in\P_n}P_i
	\end{equation}
for a certain set $B\subset E^n$.	We will be done when we can achieve the above for $B=B_{n-2,n}$. Whenever we apply \cref{lem:eps neighborhood,lem:compression along one dir} below, we replace some boxes from $\{P_i\}_{i=0}^m$ with the boxes provided by the lemmas.
	
	First we fix an $(n-2)$-dimensional face $F$ of $[0,1]^n$, say $F=[0,1]^{n-2}\times\{0\}^2$, and show how to get~\eqref{eqn:B covering condition n} with $B=F$. We set $P_0:=[0,\tfrac1{n-2}]^{n-3}\times\{0\}^3$. Invoking \cref{lem:eps neighborhood} for $A=P_0$, we get~\eqref{eqn:B covering condition n} for $B=[0,\tfrac1{n-2}]^{n-3}\times[0,\eps_1]\times\{0\}^2$, where $0<\eps_1<\tfrac1{n-2}$. With this $\eps_1$ and $k=n-2$, we can apply \cref{lem:compression along one dir} to $[0,\tfrac1{n-2}]^{n-3}\times[0,1]\times\{0\}^2$, which, by construction, is the union of $n-2$ boxes from $\{P_i\}_{i=1}^m$. This yields~\eqref{eqn:B covering condition n} for $B=[0,\tfrac1{n-2}]^{n-3}\times[0,1]\times\{0\}^2$. Next we apply \cref{lem:eps neighborhood} to $A=[0,\tfrac1{n-2}]^{n-3}\times[0,1]\times\{0\}^2$ and obtain~\eqref{eqn:B covering condition n} for $B=[0,\tfrac1{n-2}]^{n-4}\times[0,\tfrac1{n-2}+\eps_2]\times[0,1]\times\{0\}^2$ for some $\eps_2>0$. Invoke \cref{lem:compression along one dir} for this $\eps_2$ and $k=n-3$ to the box $[0,\tfrac1{n-2}]^{n-4}\times[\tfrac1{n-2},1]\times[0,1]\times(0,0)$, which, by construction, is the union of $(n-3)(n-2)$ boxes from $\{P_i\}_{i=1}^m$ (none of these boxes were modified until now). This leads to~\eqref{eqn:B covering condition n} for $B=[0,\tfrac1{n-2}]^{n-4}\times[0,1]^2\times(0,0)$. Proceeding in this manner, subsequently applying \cref{lem:eps neighborhood,lem:compression along one dir} decreasing $k$ by $1$ at each step, we arrive at~\eqref{eqn:B covering condition n} for $B=F$.
	
	Now suppose that we already established~\eqref{eqn:B covering condition n} with $B=B'$ being the union of some $(n-2)$-dimensional faces of $[0,1]^n$, but $B'\ne B_{n-2,n}$. Then we can pick an $(n-2)$-dimensional face $F'$ of $[0,1]^n$ which is not contained in $B'$ and has some $(n-3)$-dimensional face in common with some $(n-2)$-dimensional face inside $B'$. Without loss of generality, $F'=\{0\}^2\times[0,1]^{n-2}$ and the common face is $\{0\}^3\times[0,1]^{n-3}\subset B'$. We proceed similarly to the previous paragraph: first use \cref{lem:eps neighborhood} for $A=\{0\}^3\times[0,1]^{n-3}$ which establishes~\eqref{eqn:B covering condition n} for $B=\{0\}^2\times[0,\eps]\times[0,1]^{n-3}$ and some $\eps>0$, and then invoke \cref{lem:compression along one dir} for this $\eps$, $k=3$ and $F'$ to arrive at~\eqref{eqn:B covering condition n} for $B=B'\cup F'$.
	
	Extending $B$ as above one $(n-2)$-dimensional face of $[0,1]^n$ at a time, in finitely many steps we get the desired~\eqref{eqn:B covering condition n} for $B=B_{n-2,n}$.
%\end{proof}

\section{Proof of \Cref{thm:main}}\label{sec:4 upper}

%\begin{proof}
	Now we prove $C_4\le96$. First we construct a specific cover to show that $C(B_{2,4},\P_4^*)\le 88$ and then modify that cover to establish $C(B_{2,4},\P_4)\le 89$ and $C_4=C(\A_4\cup B_{2,4},\P_4)\le 96$. In this section, for simplicity, we will omit ``of $[0,1]^4$'' when referring to faces of various dimension of the $4$-cube $[0,1]^4$. In particular, we simply say vertices, the term ``edges'' will be used for $1$-dimensional faces, ``faces'' are reserved for $2$-dimensional faces, and facets are $3$-dimensional faces.

	We begin with a cover satisfying $C(B_{2,4},\P_4^*)\le 88$. For each vertex, we take the box from $\P_4^*$ with all side lengths equal to $\tfrac14$ containing the vertex and lying entirely in $[0,1]^4$. For each edge $E$, we take a box from $\P_4^*$ which is the image of $[\tfrac14,\tfrac34]\times[0,\tfrac14]^2\times\{0\}$ under a certain symmetry of the cube $[0,1]^4$ mapping the edge joining $(0,0,0,0)$ and $(1,0,0,0)$ to $E$. There are three choices of such a symmetry: of three faces containing $E$, two faces will intersect with the box by a $\tfrac12$ by $\tfrac14$ rectangle, and one face will intersect with the box only by the edge $E$. In such faces where the intersection is $E$, we will call the edge $E$ special. Otherwise, the edge will be referred to as normal. %The specific choices of covering boxes corresponding to the edges will be carried out shortly.

We aim to obtain faces of two types, as illustrated in Figure~1.

\begin{pspicture}(-3,-2)(3,4)
\psaxes[linewidth=0.3 pt,labels=none,ticks=none]{->}(0,0)(-0.5,-0.5)(3.5,3.5)
\pspolygon[linewidth=1pt](0,0)(0.75,0)(0.75,0.75)(0,0.75)
\pspolygon[linewidth=1pt](2.25,0)(3,0)(3,0.75)(2.25,0.75)
\pspolygon[linewidth=1pt](0,2.25)(0.75,2.25)(0.75,3)(0,3)
\pspolygon[linewidth=1pt](2.25,2.25)(3,2.25)(3,3)(2.25,3)
\psline[linewidth=1pt](0.75,0)(2.25,0)
\psline[linewidth=1pt](0.75,1.5)(2.25,1.5)
\psline[linewidth=1pt](0.75,3)(2.25,3)
\psline[linewidth=1pt](0,0.75)(0,2.25)
\psline[linewidth=1pt](0.75,0.75)(0.75,2.25)
\psline[linewidth=1pt](2.25,0.75)(2.25,2.25)
\psline[linewidth=1pt](3,0.75)(3,2.25)
\psline[linestyle=dashed,linecolor=darkgray,linewidth=0.3 pt](0,1.5)(0.75,1.5)
\uput{0.1pt}[-90](0.75,-0.15){$\frac{1}{4}$}
\uput{0.1pt}[-90](2.25,-0.15){$\frac{3}{4}$}
\uput{0.1pt}[-180](-0.1,0.75){$\frac{1}{4}$}
\uput{0.1pt}[-180](-0.1,2.25){$\frac{3}{4}$}
\uput{0.1pt}[-180](-0.1,1.5){$\frac{1}{2}$}
\uput{0.1pt}[-90](1.5,-1.5){$\text{Type}~A$}
\psaxes[linewidth=0.3 pt,labels=none,ticks=none]{->}(7,0)(6.5,-0.5)(10.5,3.5)
\pspolygon[linewidth=1pt](7,0)(7.75,0)(7.75,0.75)(7,0.75)
\pspolygon[linewidth=1pt](9.25,0)(10,0)(10,0.75)(9.25,0.75)
\pspolygon[linewidth=1pt](7,2.25)(7.75,2.25)(7.75,3)(7,3)
\pspolygon[linewidth=1pt](9.25,2.25)(10,2.25)(10,3)(9.25,3)
\psline[linewidth=1pt](7.75,0)(9.25,0)
\psline[linewidth=1pt](7.75,0.75)(9.25,0.75)
\psline[linewidth=1pt](7.75,2.25)(9.25,2.25)
\psline[linewidth=1pt](7.75,3)(9.25,3)
\psline[linewidth=1pt](7,0.75)(7,2.25)
\psline[linewidth=1pt](7.75,0.75)(7.75,2.25)
\psline[linewidth=1pt](9.25,0.75)(9.25,2.25)
\psline[linewidth=1pt](10,0.75)(10,2.25)
\uput{0.1pt}[-90](7.75,-0.15){$\frac{1}{4}$}
\uput{0.1pt}[-90](9.25,-0.15){$\frac{3}{4}$}
\uput{0.1pt}[-180](6.9,0.75){$\frac{1}{4}$}
\uput{0.1pt}[-180](6.9,2.25){$\frac{3}{4}$}
\uput{0.1pt}[-90](8.5,-1.5){$\text{Type}~B$}
\end{pspicture}
\begin{center}
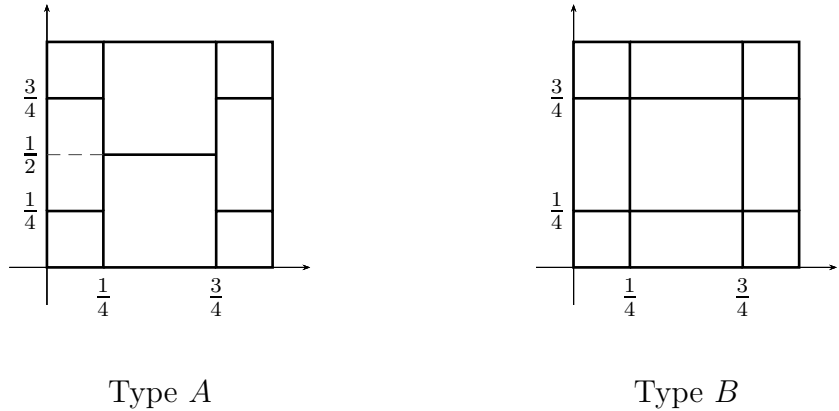
	Figure~1. Decomposition of faces by boxes of the partition
\end{center}

In type $A$ face the horizontal edges are special and the vertical ones are normal, while in type $B$ face all edges are normal. For each edge we want to pick exactly one of the three faces	containing the edge in which this edge will be special (normal in the other two faces) so that the faces become either type~A or type~B. While there are many such choices, we will describe a specific one which will be convenient for the second part of the proof. For each face of type~A, it suffices to indicate which direction is parallel to the two special edges. We use $x_j$, $j=1,\dots,4$ as the coordinate axes in $\E^4$. In each of the two facets $x_4=0$ and $x_4=1$ we assign all faces to be type~A and directions (as vectors in $(x_1,x_2,x_3)$) of special edges for each face to be as follows (here a single equation $x_i=t$, $t\in\{0,1\}$, $i=1,2,3$, describes a face as we already fixed $x_4$):
	\begin{align*}
		x_1=0 \qtq{and} x_1=1: & \quad (0,0,1),\\
		x_2=0 \qtq{and} x_2=1: & \quad (1,0,0),\\
		x_3=0 \qtq{and} x_3=1: & \quad (0,1,0).
	\end{align*}
	There are four more faces which will be assigned type~A. Each such face is given by $x_2=t$ and $x_3=s$, where $t,s\in\{0,1\}$, the direction of the special edges is always $(0,0,0,1)$. All faces that were not assigned type~A are assigned type~B. It is easy to verify that the described configuration satisfies our requirements, with $16$ faces of type~A and $8$ faces of type~B, and each edge being special in exactly one of the three containing faces. Overall, this forms a cover of $B_{2,4}$ with $16$ boxes corresponding to the vertices, $32$ boxes corresponding to the edges, $2\cdot 16$ additional boxes for faces of type~A and $8$ additional boxes for faces of type~B, providing $88$ boxes in total, as required for $C(B_{2,4},\P_4^*)\le 88$. We remark that in this cover no two boxes have overlapping interior. We need to leave this cover alone for some time.
	
%	[[it could be handy to have notations for the boxes corresponding to vertices, edges, faces of each type]]
	
	By $[\a,\b]:=\{t\a+(1-t)\b,t\in[0,1]\}$ we denote the line segment joining two points $\a, \b\in\E^4$. For $k=0,1,2$ we refer to the image of $[(\tfrac14,\tfrac k4,0,0),(\tfrac34,\tfrac k4,0,0)]$ under some symmetry of $[0,1]^4$ as side, semi-central and central segment, respectively. Our next goal is to show that for any $A_4\in\A_4$ there exists a segment $I$ which is either side, semi-central or central segment such that $C(A_4\cup I,\P_4)\le 8$.	We have
	\begin{equation*}
	A_4=\bigcup_{i=1}^4\{\x_i,\x_i+\e_i\}: \x_i,\x_i+\e_i\in [0,1]^4,
	\end{equation*}
	so $\x_i$ has $i$-th coordinate equal to zero, and let us refer to the value of the remaining three coordinates as non-trivial coordinates. Suppose there is $i$ such that $\x_i$ has all non-trivial coordinates in $[0,\tfrac14)\cup (\tfrac34,1]$. By symmetry, we can assume $i=4$,  $\x_4=(q,s,t,0)$ and $\tfrac14> q\ge s\ge t\ge0$, implying $s+t<\tfrac14+q$. Then the box $[q,\tfrac34]\times[0,s]\times[0,t]\times\{0\}$ is from $P_4$ and contains both $\x_4$ and the side segment $I:=[(\tfrac14,0,0,0),(\tfrac34,0,0,0)]$. We trivially cover the remaining $7$ points from $A_4$ using a box per point and get $C(A_4\cup I,\P_4)\le 8$.
	
	So in what follows we can assume that for every $i$ at least one of the non-trivial coordinates of $\x_i$ belongs to $[\tfrac14,\tfrac34]$. Suppose there is $i$ for which we have exactly one such coordinate. We can assume $i=4$, $\x_4=(q,s,t,0)$, $q\in[\tfrac14,\tfrac34]$ and $s,t\in[0,\tfrac14)$. Then the box $[\tfrac14,\tfrac34]\times[0,s]\times[0,t]\times\{0\}$ is from $P_4$ and contains both $\x_4$ and the side segment $[(\tfrac14,0,0,0),(\tfrac34,0,0,0)]$. We conclude as in the previous case. Next, if for some $i$ we have exactly two non-trivial coordinates of $\x_i$ in $[\tfrac14,\tfrac34]$, then assuming $i=4$, $\x_4=(q,s,t,0)$, $q,s\in[\tfrac14,\tfrac12]$ (by further application of symmetry) and $t\in[0,\tfrac14)$, the desired box will be $[\tfrac14,\tfrac34]\times[s,\tfrac12]\times[0,t]\times\{0\}$ and the (central) segment is $[(\tfrac14,\tfrac12,0,0),(\tfrac34,\tfrac12,0,0)]$.
	
	The above considerations allow us to assume that for every $i$ each non-trivial coordinate of $\x_i$ is in $[\tfrac14,\tfrac34]$. If there exists $i$ such that at least one non-trivial coordinate of $\x_i$ is in $[\tfrac14,\tfrac38)\cup(\tfrac58,1]$, then we can assume $i=4$, $\x_4=(q,s,t,0)$, $q,s\in[\tfrac14,\tfrac12]$ and $t\in[\tfrac14,\tfrac38)$. If $s<\frac38$, we take $[\tfrac14,\tfrac34]\times [\tfrac14,s]\times[0,t]\times\{0\}$ and the semi-central segment $[(\tfrac14,\tfrac14,0,0),(\tfrac34,\tfrac14,0,0)]$. If $s\ge\frac38$, we take $[\tfrac14,\tfrac34]\times [s,\tfrac12]\times[0,t]\times\{0\}$ and the central segment $[(\tfrac14,\tfrac12,0,0),(\tfrac34,\tfrac12,0,0)]$.
	
	So, we can assume that for every $i$ each non-trivial coordinate of $\x_i$ is in $[\tfrac38,\tfrac58]$. This situation is more complicated because it may require that we use two boxes covering respective $\x_i$-s together with a suitable segment. If there is at least one value of $i$ such that $\x_i$ has a non-trivial coordinate different from $\tfrac38$ or from $\tfrac58$, then we can assume $i=4$, $\x_4=(q,s,t,0)$, $q\in(\tfrac38,\tfrac12]$. Let $\x_3=(q',s',0,t')$, recall that we have $s,q',s'\in[\tfrac38,\tfrac58]$, and by symmetry we can assume $t,t'\le\tfrac12$. We will present two boxes from $\P_4$ that cover $\{\x_3,\x_4\}\cup I$, $I:=[(\tfrac14,\tfrac12,0,0),(\tfrac34,\tfrac12,0,0)]$, then the remaining six points form $A_4$ can be each covered by a box from $\P_4$ leading to the desired $C(A_4\cup I,\P_4)\le8$. It is convenient to denote $J(u):=[\min(\tfrac12,u),\max(\tfrac12,u)]$, $u\in[\tfrac38,\tfrac58]$. Then the length of $J(u)$ is $|\tfrac12-u|\le\tfrac18$. If $q'\le\tfrac12$, then we take $[q,\tfrac34]\times J(s)\times [0,t] \times\{0\}$ and $[\tfrac14,\tfrac12]\times J(s')\times \{0\}\times [0,t']$. If $q'>\tfrac12$, then we take $[\tfrac14,\tfrac12]\times J(s)\times [0,t] \times\{0\}$ and $[\tfrac12,\tfrac34]\times J(s')\times \{0\}\times [0,t']$. At last, we can assume that for every $i$ any non-trivial coordinate is either $\tfrac38$ or $\tfrac58$. In this case, we can assume $\x_3=(q',s',0,\tfrac38)$, $\x_4=(q,s,\tfrac38,0)$, $q'\le q$, and then simply take $[q',\tfrac34]\times J(s')\times \{0\}\times [0,\tfrac38]$ and $[\tfrac14,q]\times J(s)\times [0,\tfrac38] \times \{0\}$ that cover $\{\x_3,\x_4\}\cup I$.
	
	Now we can return to the cover fulfilling $C(B_{2,4},\P_4^*)\le88$ that we constructed earlier. We also have $C(A_4\cup I,\P_4)\le 8$. Using a symmetry, if needed, we can make the following assumption. If $I$ is a side segment, then $I=[(\tfrac14,0,0,0),(\tfrac34,0,0,0)]$. If $I$ is a semi-central segment, then $I=[(0,0,\tfrac14,\tfrac14),(0,0,\tfrac34,\tfrac14)]$. If $I$ is a central segment, then $I=[(\tfrac14,0,\tfrac12,0),(\tfrac34,0,\tfrac12,0,0)]$. Let $\{P_i\}_{i=1}^{96}$ be the boxes from $\P_4^*$ we constructed above to cover $B_{2,4}\cup A_4\cup I$. Eight of them are already in $\P_4$, and similarly to what was done in \cref{sec:general upper} we will perform a sequence of modifications of the collection $\{P_i\}_{i=1}^{96}$. At each step, we will have
	\begin{equation}\label{eqn:B covering condition 4}
	B_{2,4}\cup A_4 \subset \bigcup_{i=1}^{96}P_i
	\qtq{and}
	B\cup A_4\subset \bigcup_{1\le i\le 96,\, P_i\in\P_4}P_i,
	\end{equation}
	with the goal of reaching the above for $B=B_{2,4}\cup A_4$. Initially, we have~\eqref{eqn:B covering condition 4} for $B$ being a neighborhood of $I$, after application of \cref{lem:eps neighborhood} to $I$.
	
	By $B_{2,4}^*$ we denote the closure of the set obtained by removing from $B_{2,4}$ the union of all boxes of the cover corresponding to the vertices. (This removes four ``corner'' squares with side length $\tfrac14$ from each face.) We will modify these boxes in the very end of the procedure and now we focus on covering $B_{2,4}^*$. The idea is to obtain~\eqref{eqn:B covering condition 4} for $B=B_{2,4}^*$ by adding one ``cornerless'' face at a time.
	
	Each box of the cover corresponding to an edge will be modified in a specific way, reducing the total sum of the dimensions of the box. Namely, when we are adding a face in which this edge is normal, say, without loss of generality, when the face is $x_3=x_4=0$, the edge is $x_2=x_3=x_4=0$, and the original box is $[\tfrac14,\tfrac34]\times [0,\tfrac14]\times\{0\}\times[0,\tfrac14]$, then modified box will be $[\tfrac14,\tfrac34]\times [0,\tfrac14-3\eps]\times[0,\eps]\times[0,\tfrac14+\eps]$ with sufficiently small $\eps>0$. We will not make further changes to this box when other faces containing this edge are added to the union.
	
	When a face of type~A is added, there will be two possible situations. First one: a neighborhood of the corresponding central segment is already covered (parallel to the special edges). This can happen only when the first face is added, so the face is $x_2=x_4=0$ and the central segment is $[(\tfrac14,0,\tfrac12,0),(\tfrac34,0,\tfrac12,0,0)]$. Then we replace the two boxes corresponding to the face which are $[\tfrac14,\tfrac34]\times\{0\}\times[0,\tfrac12]\times\{0\}$ and $[\tfrac14,\tfrac34]\times\{0\}\times[\tfrac12,1]\times\{0\}$ with $[\tfrac14-\eps,\tfrac34+\eps]\times\{0\}\times[0,\tfrac12-3\eps]\times\{0\}$ and $[\tfrac14-\eps,\tfrac34+\eps]\times\{0\}\times[\tfrac12+3\eps,1]\times\{0\}$, for sufficiently small $\eps>0$. The second (more typical) situation: a neighborhood of a side segment contained in a special edge is already covered. Let, without loss of generality, the face be $x_2=x_4=0$ and the side segment be $[(\tfrac14,0,0,0),(\tfrac34,0,0,0,0)]$. Then we replace $[\tfrac14,\tfrac34]\times\{0\}\times[0,\tfrac12]\times\{0\}$ and $[\tfrac14,\tfrac34]\times\{0\}\times[\tfrac12,1]\times\{0\}$ with $[\tfrac14-\eps,\tfrac34+\eps]\times\{0\}\times[6\eps,\tfrac12+3\eps]\times\{0\}$ and $[\tfrac14-\eps,\tfrac34+\eps]\times\{0\}\times[\tfrac12+3\eps,1]\times\{0\}$ for sufficiently small $\eps>0$. In either of the situations, we can modify any of the not yet modified boxes corresponding to the normal edges of this face and then add the face to the union.
	
	When a face of type~B is added, one of the boxes corresponding to the edges will be already modified or a neighborhood of a semi-central segment in this face will be covered. In either of situations, without loss of generality, we can assume that a neighborhood of $[(0,0,\tfrac14,\tfrac14),(0,0,\tfrac34,\tfrac14)]$ is covered and the face is $x_1=x_2=0$. We replace $\{0\}^2\times [\tfrac14,\tfrac34] \times [\tfrac14,\tfrac34]$ with $\{0\}^2\times [\tfrac14-\eps,\tfrac34+\eps] \times [\tfrac14+4\eps,\tfrac34+\eps]$ for sufficiently small $\eps>0$ and modify the not yet modified boxes corresponding to the sides of this face.
	
	Now we describe how to add all faces to the union using the above operations. If $I$ is a side or a central segment, we can begin with the face $x_2=x_4=0$. Note that if the union already contains a face with a normal edge, then our operations allow to add another face in which this edge is special. So, we can add all faces from the facet $x_4=0$ for instance in the following order: $x_1=0$, $x_1=1$, $x_3=0$, $x_3=1$, $x_2=1$. Next we can add all faces having one side in $x_4=0$ and the other side in $x_4=1$, in arbitrary order. Such faces are either type~B with the box corresponding to the edge belonging to $x_4=0$ already modified or type~A with one of the special sides in $x_4=0$ which is already covered. We conclude by adding the faces from the facet $x_4=1$ in the same order as was done for $x_4=0$. Now if $I$ is a semi-central segment, we can begin with the face $x_1=x_2=0$, proceed with adding $x_2=x_4=0$ and then we can follow the order of adding faces as in the previous case when $I$ was a side or a central segment. The only difference will be that we simply skip the face $x_1=x_2=0$. 	

	So, we obtained~\eqref{eqn:B covering condition 4} for $B=B_{2,4}^*$. Now we apply \cref{lem:eps neighborhood} for $A=B_{2,4}^*$ and then modify each of the boxes corresponding to the vertices so that the resulting boxes still contain the vertices and have all side lengths equal to $\tfrac14-\eps$, for sufficiently small $\eps>0$ yielding~\eqref{eqn:B covering condition 4} for $B=B_{2,4}$. This completes the proof.
%\end{proof}

{\bf Conflict of interest.} On behalf of all authors, the corresponding author states
that there is no conflict of interest.

{\bf Acknowledgment.} The authors are grateful to the anonymous referee for the valuable suggestions that improved the paper.

\end{document}